\documentclass[12pt]{amsart}
\usepackage{amssymb}
\usepackage{amsfonts}
\usepackage{amssymb,latexsym}
\usepackage{enumerate}
\usepackage{mathrsfs}
\makeatletter
\@namedef{subjclassname@2010}{%
  \textup{2010} Mathematics Subject Classification}
\makeatother

  [2002/01/19 v2.2g %
    AMS font definitions%
  ]
\DeclareFontFamily{U}{euf}{}
\DeclareFontShape{U}{euf}{m}{n}{%
  <5><6><7><8><9>gen*eufm%
  <10><10.95><12><14.4><17.28><20.74><24.88>eufm10%
  }{}
\DeclareFontShape{U}{euf}{b}{n}{%
  <5><6><7><8><9>gen*eufb%
  <10><10.95><12><14.4><17.28><20.74><24.88>eufb10%
  }{}

  [2002/01/19 v2.2g %
    AMS font definitions%
  ]
\DeclareFontFamily{U}{msb}{}
\DeclareFontShape{U}{msb}{m}{n}{%
  <5><6><7><8><9>gen*msbm%
  <10><10.95><12><14.4><17.28><20.74><24.88>msbm10%
  }{}

  [2002/01/19 v2.2g %
    AMS font definitions%
  ]
\DeclareFontFamily{U}{msa}{}
\DeclareFontShape{U}{msa}{m}{n}{%
  <5><6><7><8><9>gen*msam%
  <10><10.95><12><14.4><17.28><20.74><24.88>msam10%
  }{}

\newtheorem{theorem}{Theorem}[section]

\newtheorem{proposition}[theorem]{Proposition}

\newtheorem{corollary}[theorem]{Corollary}

\theoremstyle{definition}
\newtheorem{remark}[theorem]{Remark}
\newtheorem{definition}[theorem]{Definition}

\numberwithin{equation}{section} 

\def\C{\mathbb C_p}
\def\BZ{\mathbb Z}
\def\Z{\mathbb Z_p}
\def\Q{\mathbb Q_p}
\def\C{\mathbb C_p}
\def\BZ{\mathbb Z}
\def\Z{\mathbb Z_p}
\def\Q{\mathbb Q_p}

\begin{document}

\title[$p$-adic analogue of Weil's elliptic functions]
{On $p$-adic analogue of Weil's elliptic functions according to Eisenstein}

\author{Su Hu and Min-Soo Kim}

\address{ Department of Mathematics and Statistics, McGill University, 805 Sherbrooke St. West, Montr\'eal, Qu\'ebec, H3A 2K6, Canada}
\email{hus04@mails.tsinghua.edu.cn, hu@math.mcgill.ca}

\address{Division of Cultural Education, Kyungnam University,
7(Woryeong-dong) kyungnamdaehak-ro, Masanhappo-gu, Changwon-si,
Gyeongsangnam-do 631-701, South Korea }
\email{mskim@kyungnam.ac.kr}





\begin{abstract}
In this paper, using  $p$-adic integration with values in spaces of
modular forms, we construct the $p$-adic analogue of  Weil's elliptic functions according to Eisenstein in the  book ``Elliptic functions according to Eisenstein and and Kronecker". This construction extends
Serre's $p$-adic family of Eisenstein series in ``Formes modulaires et fonctions z\^{e}ta $p$-adiques". We show that the power series expansion  of Weil's elliptic functions also exists in the $p$-adic case.
\end{abstract}

\maketitle

\def\C{\mathbb C_p}
\def\BZ{\mathbb Z}
\def\Z{\mathbb Z_p}
\def\Q{\mathbb Q_p}


\section{Introduction}\label{introduction}

Throughout this paper, we use the following notations.
\begin{equation*}
\begin{aligned}
\qquad \mathbb{C}  ~~~&- ~\textrm{the field of complex numbers}.\\
\qquad p  ~~~&- ~\textrm{an odd rational prime number}. \\
\qquad\mathbb{Z}_p  ~~~&- ~\textrm{the ring of $p$-adic integers}. \\
\qquad\mathbb{Q}_p~~~&- ~\textrm{the field of fractions of}~\mathbb Z_p.\\
\qquad\mathbb C_p ~~~&- ~\textrm{the completion of a fixed algebraic closure}~\overline{\mathbb Q}_p~ \textrm{of}~\mathbb Q.\\
\qquad v_p ~~~&- ~\textrm{the $p$-adic valuation of}~\mathbb
C_p~\textrm{normalized so that}~ |p|_p=p^{-v_p(p)}=p^{-1}.
\end{aligned}
\end{equation*}

Suppose $W$ is a lattice in the complex plane, $\omega_{1}$ and $\omega_{2}$ are two generators of $W$, so that $W$ consists of the points $w=m_{1}\omega_{1}+m_{2}\omega_{2}$, where $m_{1}$ and $m_{2}$ are integers. In his historical book ``Elliptic functions according to Eisenstein and Kronecker"~\cite{Weil}, generalizing the Hurwitz zeta functions~\cite[p. 6]{Weil}, A. Weil~\cite[p. 14]{Weil} introduced the following elliptic function:
\begin{equation}~\label{weil}
E_{k}(x,W)=\sum_{w\in W}\frac{1}{(x+w)^{k}}=\sum_{(m_{1},m_{2})\in\mathbb{Z}^{2}}\frac{1}{(x+m_{1}\omega_{1}+m_{2}\omega_{2})^{k}}
\end{equation}
for $x\in\mathbb{R}_{> 0}$, which is also a generalization of the homogeneous Eisenstein series defined by
\begin{equation}~\label{Eisenstein}
G_{k}(W)=\sum_{0\neq w\in W}\frac{1}{w^{k}}=\sum_{(m_{1},m_{2})\in\mathbb{Z}^2\backslash(0,0)}\frac{1}{(m_{1}\omega_{1}+m_{2}\omega_{2})^{k}}.
\end{equation}
Let $\tau=\frac{\omega_{1}}{\omega_{2}}$, then
\begin{equation*}
G_{k}(\tau)=\omega_{2}^{-k}\sum_{(m_{1},m_{2})\in\mathbb{Z}^2\backslash(0,0)}\frac{1}{(m_{1}\frac{\omega_{1}}{\omega_2}+m_{2})^{k}}=\omega_{2}^{-k}\sum_{(m_{1},m_{2})\in\mathbb{Z}^2\backslash(0,0)}
\frac{1}{(m_{1}\tau+m_{2})^{k}}.
\end{equation*}

For $s\in\mathbb{C}, ~\mathfrak{R}(s) > 1$, the Riemann zeta function is defined as follows:
\begin{equation}~\label{Riemann}\zeta_{\mathbb{Q}}(s)=\sum_{n=1}^{\infty}\frac{1}{n^{s}},\end{equation}
for $x\in\mathbb{R}_{> 0}$ and $s\in\mathbb{C}$ with $\mathfrak{R}(s) > 1$, the Hurwitz zeta function is defined as follows:
\begin{equation}~\label{Hurwitz}\zeta(s,x)=\sum_{n=0}^{\infty}\frac{1}{(n+x)^{s}}.\end{equation}

\begin{remark}\label{remark} Comparing  (\ref{weil}) with (\ref{Hurwitz}), we may call (\ref{weil}) the Hurwitz-type Eisenstein series.\end{remark}

We have (see ~\cite[p. 20]{Weil}):
\begin{equation}~\label{weil2}
E_{k}(x,W)=\frac{1}{x^{k}}+(-1)^{k}\sum_{m=1}^{\infty}\binom{2m-1}{k-1}e_{2m}(W)x^{2m-k},
\end{equation}
where $$e_{2m}(W)=\frac{(2\pi i/\omega_{1})^{2m}}{(2m-1)!}\Big((-1)^{m}\frac{B_{m}}{2m}+2\sum_{N=1}^{\infty}\sigma_{2m-1}(N)q^{N}\Big),$$ and  $\sigma_{2m-1}(N)=\sum_{d|N}d^{2m-1}.$

Now we go to the $p$-adic situation.

 The counterparts of Eisenstein series ~(\ref{Eisenstein}) in the $p$-adic world
were defined by Serre in \cite[p. 205, Sec.1.6]{Serre}. Serre defined $p$-adic modular forms
as $p$-adic $q$-expansions which are uniform limits of $q$-expansions of classical modular
forms. The $p$-adic Eisenstein series are parametrized by $X=\mathbb{Z}_{p}\times\mathbb{Z}/(p-1)\mathbb{Z}$.

Here we  introduce some basic notations of $p$-adic modular form \`{a} la Serre~\cite{Serre}.

First, we recall the definition of classical Eisenstein series.

Let $k$ be an even integer and $\tau$  a complex number with
strictly positive imaginary part. Define the Eisenstein series
$G_{k}(\tau)$ of weight $k$, by the following series:

$$G_{k}(\tau) = \sum_{ (m,n)\in\mathbb{Z}^2\backslash(0,0)}
\frac{1}{(m\tau+n)^{k}}.$$

This series  absolutely converges  to a holomorphic function of
$\tau$ in the upper half-plane  and its Fourier expansion given
below shows that it extends to a holomorphic function at $\tau=
i\infty$. If $k\geq 4$ and $$\left(
    \begin{array}{cc}
     a & b \\
      c & d\\
    \end{array}\right)\in \textrm{SL}_{2}(\mathbb{Z}),
$$ then

$$G_{k} \left( \frac{ a\tau +b}{ c\tau + d} \right) = (c\tau
+d)^{k} G_{k}(\tau),$$ thus $G_{k}(\tau)$ is a modular
form of weight $k$.

Let $q=e^{2\pi i\tau}$. Then the Fourier series of the Eisenstein
series is

\begin{equation}\label{Gk}G_{k}(\tau) = -\frac{B_{k}}{2k}+\sum_{n=1}^{\infty}
\sigma_{k-1}(n)q^{n}.\end{equation}

Here, $B_{k}$ is the $k$th Bernoulli number and
$\sigma_{k-1}(n)=\sum_{d|n}d^{k-1}$.

Now we pass to the $p$-adic limit. Define, for $k\in
X=\mathbb{Z}_{p}\times\mathbb{Z}/(p-1)\mathbb{Z}$ and
$n\geq 1$: $\sigma_{k-1}^{*}(n)=\sum_{\substack{d|n\\
(p,d)=1}}d^{k-1}$. (See \cite[p. 205]{Serre}). If $k$ is even, there exists a sequence of even
integers $\{k_{i}\}_{i=1}^{\infty}$ such that
$|k_{i}|\rightarrow\infty$ and $k_{i} \rightarrow k$ when
$i\to\infty$. Then the sequence $G_{k_{i}}=
-\frac{B_{k_{i}}}{2k_{i}}+\sum_{n=1}^{\infty}
\sigma_{k_{i}-1}(n)q^{n}$ has a limit:
\begin{equation}\label{GK*}G_{k}^{*}=a_{0}+\sum_{n=1}^{\infty} \sigma_{k-1}^{*}(n)q^{n}\end{equation}
with
$a_{0}=\frac{1}{2}\lim_{i\to\infty}\zeta(1-k_{i})\equiv
\frac{1}{2}\zeta^{*}(1-k)$.

The function $\zeta^{*}$ is thus defined on the odd elements of
$X\setminus\{1\}$.

We have the following result on $\zeta^{*}$.

\begin{theorem}[see Serre {\cite[p. 206, Theorem 3]{Serre}}]~\label{serre} Let $\chi$ be a character
on $\mathbb{Z}_{p}$ and let $L_{p}(s,\chi)$ be its the $p$-adic $L$-function. If $(s,u)\not=1$ is an odd element of
$X=\mathbb{Z}_{p}\times\mathbb{Z}/(p-1)\mathbb{Z},$ then
$$\zeta^{*}(s,u)=L_{p}(s,\omega^{1-u}),$$ where $\omega$ is the Teichm\"uller character.\end{theorem}

For $k=(s,u)\in X$ and $u$ is even, the coefficients of
$G_{k}^{*}=G_{s,u}^{*}$ are given by
\begin{equation}\label{eq1}
\begin{aligned}
a_{0}(G_{s,u}^{*})&=\frac{1}{2}\zeta^{*}(1-s,1-u)=\frac{1}{2}L_{p}(1-s,\omega^{u}),\\
a_{n}(G_{s,u}^{*})&=\sum_{\substack{d|n\\
(p,d)=1}} d^{-1}\omega(d)^{u}\langle d
\rangle^{s}\end{aligned}\end{equation}(see \cite[p.~245]{Serre}).

Thus the assignment
$$(s,u)\mapsto G_{s,u}^{*}$$
gives a family of $p$-adic modular forms parametrized by the group of weights $X$.

In the $p$-adic world, the counterpart of the Hurwitz zeta functions~(\ref{Hurwitz}) was first defined by Washington on $\mathbb{Q}_{p}\setminus\mathbb{Z}_{p}$ in his book~\cite{Wa2} by using the power series expansions, that is
$$H_{p}(s,a,F)=\frac{1}{s-1}\frac{1}{F}\langle a \rangle^{1-s}\sum_{j=0}^{\infty}\binom{1-s}{j}(B_{j})\left(\frac{F}{a}\right)^{j},$$
where $|s|_{p}<R_{p}=p^{(p-2)/(p-1)},$ $a$ and $F$ are integers with $0 < a < F$, $p\mid F$, $p\nmid a$ and $B_{j}$ is the $j$th Bernoulli number (see \cite[p.~55]{Wa2}).

 Recently, using the $p$-adic Volkenborn integral~\cite{Volkenborn1,Volkenborn2}, Cohen~\cite[Chapter 11]{Co2} and  Tangedal-Young~\cite{TP} extended the domain of definition of $p$-adic Hurwitz zeta functions to $\mathbb{C}_{p}$.

 First, we recall Tangedal-Young's~\cite{TP} definition of
$p$-adic Hurwitz zeta functions $\zeta_{p}(s,x)$ for
$x\in\mathbb{C}_{p}\setminus\mathbb{Z}_{p}$.

Let $UD(\mathbb Z_p)$ be the space of uniformly (or strictly)
differentiable function on $\mathbb Z_p.$ Then the Volkenborn
integral of $f$ is defined by
\begin{equation}\label{-q-e}
\int_{\mathbb Z_p}f(z)da
=\lim_{N\rightarrow\infty}\frac1{p^N}\sum_{a=0}^{p^N-1}f(a)
\end{equation} and this limit always exists when $f\in UD(\mathbb
Z_p)$ (see \cite[p.~264]{Ro})

The projection function $\langle x\rangle$ is defined for all
$x\in \mathbb{C}_p^{\times}$, as was done by Kashio~\cite{Kashio}
and by  Tangedal-Young in~\cite{TP}, for details, see Section~\ref{a0} below.
We also define
$\omega_v(\cdot)$ on $\mathbb{C}_{p}^{\times}$ by
$\omega_v(x)={x}/{\langle x\rangle}.$

\begin{definition}[{see Tangedal and Young ~\cite[p.~1248, (3.2)]{TP}}]\label{p-E-zeta}
For $x\in \mathbb{C}_{p}\backslash \mathbb{Z}_{p}$, we define the
$p$-adic Hurwitz zeta function $\zeta_{p}(s,x)$ by the following
formula
\begin{equation}~\label{pz1}\zeta_{p}(s,x)=\frac{1}{s-1}\int_{\Z}\omega_{v}(x+a)\langle x+a\rangle^{1-s}da.\end{equation}
\end{definition}
\begin{remark} Tangedal and Young defined a more general
$p$-adic Hurwitz zeta functions in the multiple case
(\cite[p.~1248]{TP}).
\end{remark}

In his book~\cite[Chapter 11]{Co2}, Cohen gave a definition of
the $p$-adic Hurwitz zeta functions $\zeta_{p}(s,x)$ for
$x\in\mathbb{Z}_{p}$.

Here we recall his definitions.

Let $\chi$ be a Dirichlet character modulo $p^{v}$ for some $v$.
We can extend the definition of $\chi$ to $\mathbb{Z}_{p}$ as in
\cite[p.~281]{Co2}, that is, if $a_{n}\in\mathbb{Z}$ and
$a_{n}$ is a sequence tending to $a$ $p$-adically, we have
$v_{p}(a_{n}-a_{m})\geq v$ for $n$ and $m$ sufficiently large, so
$\chi(a_{n})$ is an ultimately constant sequence, and we set
$\chi(a)=\chi(a_{n})$ for $v_{p}(a-a_{n})\geq v$. A character $\chi$ is called
a Dirichlet character on $\mathbb{Z}_{p}$ and $\chi$ is even if
$\chi(-1)=1$. In particular, $\omega^{u}$ are such characters.
$\omega^{u}$ is even if and only if $u$ is even.

\begin{definition}[{see Cohen~\cite[p. 291, Definition 11.2.12]{Co2}}]\label{p-E-zeta2}
Let $\chi$ be a character modulo $p^{v}$ with $v\geq 1$. If
$x\in\mathbb{Z}_{p}$ and $s\in\mathbb{C}_{p}$ such that
$|s|_{p}<R_{p}=p^{(p-2)/(p-1)}$ and $s\not=1$,  we define
\begin{equation}\label{pz2}\zeta_{p}(s,\chi,x)=\frac{1}{s-1}\int_{\mathbb{Z}_{p}}\chi(x+a)\langle x+a \rangle^{1-s}da,\end{equation}
and we will simply write $\zeta_{p}(s,x)$ instead of
$\zeta_{p}(s,\chi_{0},x)$, where $\chi_{0}$ is a trivial character
modulo $p^{v}$.\end{definition}

Let $L_{p}(s,\chi)$ be the $p$-adic $L$-function defined by Iwasawa using the method of $p$-adic interpolations. (See ~\cite[p. 29, Theorem 3]{Iwa}).
By ~\cite[p. 295, Proposition 11.2.20 (3)]{Co2}, we have $L_{p}(s,\chi)=\zeta_{p}(s,\chi,0)$. Thus Cohen's definitions of
$p$-adic Hurwitz zeta functions  extend Iwasawa's definitions of $p$-adic $L$-functions.

Let $A_{1}=\mathbb{Z}_{p}$, $A_{2}=\{\omega^{u}~|~
u~\textrm{is even}\}$, $B_{1}=\{s\in\mathbb{C}_{p}~|~|s|_{p} < R_{p}=p^{(p-2)/(p-1)}\},$
$B_{2}$ be the set of even Dirichlet characters on $\mathbb{Z}_{p}$, $B_{3}=\mathbb{Z}_{p}$, $C_{1}=\mathbb{C}_{p}$ and
$C_{2}=\mathbb{C}_{p}\setminus\mathbb{Z}_{p}$.

In conclusion,  the definitions of the $p$-adic Hurwitz zeta functions have been extended
by the following graph:

\begin{equation}\label{graph}
\begin{aligned}
\zeta^{*}(s,u),&~~ (s,u)\in X\longrightarrow \zeta_{p}(s,x),~~ (s,x)\in C_{1}\times C_{2}~(\textrm{Tangedal \& Young})\\
&\downarrow\\
L_{p}(s,\omega^{u}), &~~ (s,\omega^{u})\in A_{1}\times A_{2}\\
&\downarrow\\
L_{p}(s,\chi),&~~ (s,\chi)\in B_{1}\times B_{2}~(\textrm{Iwasawa})\\
&\downarrow\\
\zeta_{p}(s,\chi,x),&~~ (s,\chi,x)\in B_{1}\times B_{2}\times B_{3}~(\textrm{Cohen}).
\end{aligned}
\end{equation}

In this note, using $p$-adic integration with values in
spaces of modular forms, involving Cohen's~\cite[Chapter 11]{Co2} and  Tangedal-Young's~\cite{TP} definitions of $p$-adic Hurwitz zeta functions on $\mathbb{C}_{p}$,
we shall construct the  counterpart of  Weil's elliptic functions~(\ref{weil}) in the $p$-adic world. As in the complex case, this construction generalizes the definition of the $p$-adic Eisenstein series. In other words, this construction also extends the  parameter space of
$p$-adic  family of Eisenstein series (see Proposition~\ref{sg} below).

 The following parts of this paper are organized as follows.

For $k=(s,u)\in X=\mathbb{Z}_{p}\times\mathbb{Z}/(p-1)\mathbb{Z}$ with $u$ even,
 let $G^{*}_{s,u}=a_{0}(s,u)+\sum_{n=1}^{\infty}a_{n}(s,u)q^{n}$ be the $p$-adic family of Eisenstein series, where $a_{0}(s,u)=a_{0}(G_{s,u}^{*})$
 and $a_{n}(s,u)=a_{n}(G_{s,u}^{*})$ for $n\geq 1$ (see Eq. (\ref{eq1})).
 In Sections \ref{a0} and \ref{an}, we will extend the definitions of $a_{0}(s,u)$ and $a_{n}(s,u)$, $n\geq 1$, respectively.
In the last section, we show that the power series expansion  of Weil's elliptic functions (\ref{weil2})
also exists in the $p$-adic case. (See Theorem~\ref{power} below).

\section{$a_{0}(s,u)$}~\label{a0}
Denote by $A_{1}=\mathbb{Z}_{p}$ and $A_{2}=\{\omega^{u}~|~
u~\textrm{is even}\}$. For $(s,u)\in A_{1}\times A_{2}$, from
Eq.~(\ref{eq1}) in Section \ref{introduction}, we have
$a_{0}(G_{s,u}^{*})=\frac{1}{2}L_{p}(1-s,\omega^{u})$, thus the assignment
$$(s,u)\mapsto a_{0}(G_{s,u}^{*})$$
gives a family of $p$-adic functions parametrized by the group
 $A_{1}\times A_{2}$.

In this section, we extend the parameter space of
$a_{0}(s,u)=a_{0}(G_{s,u}^{*})$ to the general case.

First, we recall some definitions and results of
$p$-adic Hurwitz zeta functions.

The projection function $\langle x\rangle$ is defined for all
$x\in \mathbb{C}_p^{\times}$, as was done by Kashio~\cite{Kashio}
and by  Tangedal-Young in~\cite{TP}.

After fixing an embedding of $\bar{\mathbb{Q}}$ into $\mathbb{C}_{p}$.
Let $p^{\mathbb{Q}}$ denote the image in $\mathbb{C}_{p}^{\times}$  of
the set of positive real rational powers of $p$ under this
embedding and let $\mu$ denote the group of roots of unity in
$\mathbb{C}_{p}^{\times}$ of order not divisible by $p$. For
$x\in\mathbb{C}_{p}$, $|x|_{p}=1$, there exists a unique elements
$\hat{x}\in\mu$ such that $|x-\hat{x}|_{p} < 1$ (called the
Teichm\"uller representative of $x$); it may be defined by
$\hat{x}=\lim_{n\to\infty}x^{p^{n!}}$. We extend this definition to
$x\in\mathbb{C}_{p}^{\times}$ by
\begin{equation} \hat{x}:=(\widehat{x/p^{v_{p}(x)}}),\end{equation}
that is, we define $\hat{x}=\hat{u}$ if $x=p^{r}u$ with $p^{r}\in
p^{\mathbb{Q}}$ and $|u|_{p}=1$, then we define the function
$\langle\cdot\rangle$ on $\mathbb{C}_{p}^{\times}$ by $$\langle
x\rangle=p^{-v_{p}(x)}x/\hat{x},$$ and we also define
$\omega_v(\cdot)$ on $\mathbb{C}_{p}^{\times}$ by
$$\omega_v(x)=\frac{x}{\langle x\rangle}.$$  From this we get an
internal product decomposition of multiplicative groups
\begin{equation}
\mathbb{C}_{p}^{\times}\simeq p^{\mathbb{Q}}\times\mu\times D
\end{equation}
where $D=\{x\in\mathbb{C}_{p}: |x-1|_{p} < 1\},$ given by
\begin{equation}
x=p^{v_{p}(x)}\cdot\hat{x}\cdot\langle x\rangle\mapsto
(p^{v_{p}(x)},\hat{x},\langle x\rangle).\end{equation} As remarked
by Tangedal and Young in~\cite{TP}, this decomposition of
$\mathbb{C}_{p}^{\times}$ depends on the choice of embedding of
$\bar{\mathbb{Q}}$ into $\mathbb{C}_{p}$; the projections
$p^{v_{p}(x)},\hat{x},\langle x\rangle$ are uniquely determined up
to roots of unity. However for $x\in\mathbb{Q}_{p}^{\times}$ the
projections $p^{v_{p}(x)},\hat{x},\langle x\rangle$ are uniquely
determined and do not depend on the choice of the embedding. Notice
that the projections $x\mapsto p^{v_{p}(x)}$ and $x\mapsto \hat{x}$
are constant on discs of the form $\{x\in\mathbb{C}_{p}:|x-y|_{p} <
|y|_{p}\}$ and therefore have derivative zero whereas the
projections $x\mapsto\langle x\rangle$ has derivative
$\frac{d}{dx}\langle x\rangle=\langle x\rangle/x$.

For $x\in\mathbb{C}_{p}^{\times}$ and $s\in\mathbb{C}_{p}$ we define
$\langle x\rangle^{s}$ (~\cite[p. 141]{SC} ) by
\begin{equation} \langle
x\rangle^{s}=\sum_{n=0}^{\infty}\binom{s}{n}(\langle
x\rangle-1)^{n}\end{equation} which converges for $(x,s)$ as noted in Proposition~\ref{analytic} below.

We have the following result on the analytic properties of the $p$-adic function $\langle x\rangle^{s}$.

\begin{proposition}[{see Tangedal and Young ~\cite[p. 1245, Proposition 2.1]{TP}}]\label{analytic} For any
 $x\in\mathbb{C}_{p}^{\times}$ the function $s\mapsto\langle
 x\rangle^{s}$ is a $C^{\infty}$ function of $s$ on
 $\mathbb{Z}_{p}$ and  is an analytic function of
$s$ on the disc $|s|_{p} < p^{-1/(p-1)}|\rm{log}$$_{p}\langle x\rangle|_{p}^{-1}$; on this disc it is locally analytic as a function of $x$ and
 independent of the choice made to define the $\langle
 \cdot\rangle$ function. If $x$ lies in a finite extension $K$ of
 $\mathbb{Q}_{p}$ whose ramification index over $\mathbb{Q}_{p}$ is
 less than $p-1$, then  $s\mapsto\langle
 x\rangle^{s}$  is analytic for $|s|_{p} <
 |\pi|_{p}^{-1}p^{-1/(p-1)}$, where $(\pi)$ is the maximal ideal of
 the ring of integers $O_{K}$ of $K$. If $s\in\mathbb{Z}_{p}$, then the
 function $s\mapsto\langle
 x\rangle^{s}$ is an analytic function of $x$  on any disc of the form $\{x\in\mathbb{C}_{p}:|x-y|_{p} <
|y|_{p}\}$.
 \end{proposition}

Using the above result, Tangedal and Young proved the following
analytic properties for $\zeta_{p}(s,x)$, where the $p$-adic Hurwitz zeta function $\zeta_{p}(s,x)$ was defined by Eq.~(\ref{pz1}).

\begin{theorem}[{see Tangedal and Young  ~\cite[p. 1248, Theorem 3.1]{TP}}]
For any  choice of $x\in\mathbb{C}_{p}\backslash \mathbb{Z}_{p}$
the function $\zeta_{p}(s,x)$ is a $C^{\infty}$ function of $s$ on
$\mathbb{Z}_{p}\backslash\{1\}$, and is an analytic function of
$s$ on the disc $|s|_{p} < p^{-1/(p-1)}|\rm{log}$$_{p}\langle x\rangle|_{p}^{-1}$; on this disc it is
locally analytic as a function of $x$ and independent of the
choice made to define the $\langle\cdot\rangle$ function. If $x$
is so chosen to lie in a finite extension $K$ of $\mathbb{Q}_{p}$
whose ramification index over $\mathbb{Q}_{p}$ is less than $p-1$,
then $\zeta_{p}(s,x)$ is analytic for $|s|_{p} <
 |\pi|_{p}^{-1}p^{-1/(p-1)}$, except for a simple pole at $s=1$. If $s\in\mathbb{Z}_{p}\backslash\{1\}$, then the function
 $\zeta_{p}(s,x)$ is locally analytic as a function of $x$ on
 $\mathbb{C}_{p}\backslash\mathbb{Z}_{p}$.
\end{theorem}

Recall $C_{1}=\mathbb{C}_{p}$ and
$C_{2}=\mathbb{C}_{p}\setminus\mathbb{Z}_{p}$ as in Section~\ref{introduction}. For  $(s,x)\in C_{1}\times C_{2}$, if we  set
\begin{equation}~\label{a0sx}
a_{0}(s,x)=\frac{1}{2}\zeta_{p}(1-s,x),\end{equation} then from the
above theorem we have the following properties for $a_{0}(s,x)$.

\begin{corollary}
For any  choice of $x\in\mathbb{C}_{p}\backslash \mathbb{Z}_{p}$
the function $a_{0}(s,x)$ is a $C^{\infty}$ function of $s$ on
$\mathbb{Z}_{p}\backslash\{0\}$, and is an analytic function of
$s$ on the disc $|s-1|_{p} < p^{-1/(p-1)}|\rm{log}$$_{p}\langle x\rangle|_{p}^{-1}$; on this disc it is
locally analytic as a function of $x$ and independent of the
choice made to define the $\langle\cdot\rangle$ function. If $x$
is so chosen to lie in a finite extension $K$ of $\mathbb{Q}_{p}$
whose ramification index over $\mathbb{Q}_{p}$ is less than $p-1$,
then $a_{0}(s,x)$ is analytic for $|s-1|_{p} <
 |\pi|_{p}^{-1}p^{-1/(p-1)}$, except for a simple pole at $s=0$. If
 $s\in\mathbb{Z}_{p}\backslash\{0\}$, then
 $a_{0}(s,x)$ is locally analytic as a function of $x$ on
 $\mathbb{C}_{p}\backslash\mathbb{Z}_{p}$.
\end{corollary}

 Recall $B_{1}=\{s\in\mathbb{C}_{p}~|~|s|_{p} < R_{p}=p^{(p-2)/(p-1)}\},$
$B_{2}$ the set of even Dirichlet characters on $\mathbb{Z}_{p}$ and $B_{3}=\mathbb{Z}_{p}$ as in Section~\ref{introduction}. For  $(s,\chi,x)\in B_{1}\times B_{2}\times B_{3}$,  we can extend the definition  of $p$-adic functions
 $a_{0}(s,\chi)$ to three variables by setting
\begin{equation}\label{de1} a_{0}(s,\chi,x)=\frac{1}{2}\zeta_{p}(1-s,\chi,x),\end{equation}
 where the $p$-adic Hurwitz zeta function $\zeta_{p}(s,\chi,x)$ was defined by Eq.~(\ref{pz2}).

The following result indicates that the above definition is indeed
an extension of the original definition set by Serre.

\begin{proposition}\label{s1} Let $A_{1}=\mathbb{Z}_{p}$ and $A_{2}=\{\omega^u~|~
u~\textrm{is even}\}$. For $(s,u)\in X=A_{1}\times A_{2}$, we have
$a_{0}(s,\omega^{u},0)=a_{0}(G_{s,u}^{*}).$
\end{proposition}
\begin{proof} By Eq.~(\ref{de1}), \cite[Proposition 11.2.20 (3)]{Co2}, Eq.~(\ref{eq1}), we have
\begin{equation}\begin{aligned}a_{0}(s,\omega^{u},0)&=\frac{1}{2}\zeta_{p}(1-s,\omega^{u},0)\\&=\frac{1}{2}L_{p}(1-s,\omega^{u})\\&=a_{0}(G_{s,u}^{*}).\end{aligned}\end{equation}
\end{proof}

In conclusion, we have extended the parameter space of $a_{0}(s,u)$
by the following graph (compare with Graph~(\ref{graph}) above):

\begin{equation}\label{graph2}\begin{aligned}a_{0}(s,u),&~~ (s,u)\in
X\longrightarrow a_{0}(s,x),~~ (s,x)\in
C_{1}\times C_{2}\\&\downarrow\\  a_{0}(s,\omega^{u}),&~~ (s,\omega^{u})\in A_{1}\times A_{2}\\&\downarrow\\ a_{0}(s,\chi),&~~ (s,\chi)\in B_{1}\times B_{2}\\&\downarrow\\
a_{0}(s,\chi,x),&~~ (s,\chi,x)\in B_{1}\times B_{2}\times B_{3}
\end{aligned}\end{equation}

\section{$a_{n}(s,u)$}~\label{an}
Throughout this section, we assume that $n \geq 1$.

Let $U$ be any open subset of $\mathbb{Z}_{p}$. The $p$-adic
$\delta$-measure is defined as follows:
\begin{equation}
{\delta_{d}}(U)=\begin{cases} 1, &\textrm{if}~~~~ d\in U\\0,
&\textrm{if}~~~~ d\not\in U.\end{cases}\end{equation}

Let $\mu_{n}=\sum_{\substack{d|n\\
(p,d)=1}} d^{-1}\delta_{d}$.

Following
 Panchishkin~\cite[p. 2361]{Pan}, we have the following $p$-adic integration
 representation of $a_{n}(s,u)$.
\begin{equation}\label{inte}\begin{aligned}
a_{n}(s,u)&=\sum_{\substack{d|n\\
(p,d)=1}} d^{-1}\omega(d)^{u}\langle d\rangle^{s}\\&=\sum_{\substack{d|n\\
(p,d)=1}}d^{-1}\int_{\mathbb{Z}_{p}}\omega(a)^{u}\langle a
\rangle^{s}
d\delta_{d}(a)\\&=\int_{\mathbb{Z}_{p}}\omega(a)^{u}\langle a
\rangle^{s}d\mu_{n}
\end{aligned}\end{equation}
(also see Katz's discussions on this measure \cite[p. 496]{katz}).

In this section, by using $p$-adic integral
representation of $a_{n}(s,u)$~(see Eq. (\ref{inte})),
we extend the parameter space of $a_{n}(s,u)$ as in Graph (\ref{graph2}).

\begin{definition}\label{ansx}
For $x\in \mathbb{C}_{p}\backslash \mathbb{Z}_{p}$, we define the
  $a_{n}(s,x)$ by the following
formula
\begin{equation}\begin{aligned}\label{ansxe}a_{n}(s,x)&=\int_{\Z}\omega_{v}(x+a)\langle x+a\rangle^{s}d\mu_{n}(a)\\&=\int_{\Z}(x+a)\langle x+a\rangle^{s-1}d\mu_{n}(a).\end{aligned}\end{equation}
\end{definition}

By Proposition~\ref{analytic}, we have the following result.
\begin{theorem}~\label{an1}
For any  choice of $x\in\mathbb{C}_{p}\backslash \mathbb{Z}_{p}$
the function $a_{n}(s,x)$ is a $C^{\infty}$ function of $s$ on
$\mathbb{Z}_{p}$, and is an analytic function of
$s$ on the disc $|s-1|_{p} < p^{-1/(p-1)}|\rm{log}$$_{p}\langle x\rangle|_{p}^{-1}$; on this disc it is locally analytic
as a function of $x$ and independent of the choice made to define
the $\langle\cdot\rangle$ function. If $x$ is so chosen to lie in
a finite extension $K$ of $\mathbb{Q}_{p}$ whose ramification
index over $\mathbb{Q}_{p}$ is less than $p-1$, then $a_{n}(s,x)$
is analytic for $|s-1|_{p} <
 |\pi|_{p}^{-1}p^{-1/(p-1)}$. If $s\in\mathbb{Z}_{p}$, then the function
 $a_{n}(s,x)$ is locally analytic as a function of $x$ on
 $\mathbb{C}_{p}\backslash\mathbb{Z}_{p}$.
\end{theorem}

\begin{proposition}~\label{delta1} For any area of $(s,x)\in C_{1}\times C_{2}$ indicated in the above Theorem which $a_{n}(s,x)$ is continuous,
we have $$a_{n}(s,x)= \sum_{\substack{d|n\\
(p,d)=1}}d^{-1} \omega_{v}(x+d)\langle x+d\rangle^{s}.$$
\end{proposition}
\begin{proof} Let $d$ be a positive integer such that $d~|~n$ and $(d,p)=1$.
For any area indicated in the above Theorem, by Proposition~\ref{analytic}, the functions $\omega_{v}(x)$ and
$\langle x\rangle^{s}$ are  continuous, and by Theorem~\ref{an1}, the function
$a_{n}(s,x)$ is also continuous, thus from the definition of continuous functions, for any $\epsilon
> 0$, there exists a neighborhood $U_{d}$ of $d$, such that
$$|\omega_{v}(a)\langle a\rangle^{s}-\omega_{v}(d)\langle
d\rangle^{s}| < \epsilon,$$ for any $a\in U_{d}$, that is,
$$|\omega_{v}(a)\langle a\rangle^{s}\delta_{d}(a)-\omega_{v}(d)\langle
d\rangle^{s}| < \epsilon,$$ for any $a\in U_{d}$.

This means $$\int_{\mathbb{Z}_{p}}\omega_{v}(a)\langle
a\rangle^{s}d\delta_{d}(a)=\omega_{v}(d)\langle
d\rangle^{s},$$ and
$$\int_{\mathbb{Z}_{p}}\omega_{v}(a)\langle
a\rangle^{s}d\mu_{n}(a)=\sum_{\substack{d|n\\
(p,d)=1}}d^{-1} \omega_{v}(d)\langle d\rangle^{s}.$$ By
Eq.~(\ref{ansxe}), we have our result.
\end{proof}

\begin{definition}\label{de2}
Let $\chi$ be a character modulo $p^{v}$ with $v\geq 1$. If
$x\in\mathbb{Z}_{p}$ and $s\in\mathbb{C}_{p}$ such that
$|s|_{p}<R_{p}=p^{(p-2)/(p-1)}$,  then we define
$$a_{n}(s,\chi,x)=\int_{\mathbb{Z}_{p}}\chi(x+a)\langle x+a \rangle^{s}d\mu_{n}(a).$$
\end{definition}

Since $\chi(x)$ is a continuous function on $\mathbb{Z}_{p}$ and
the function $\langle x \rangle^{s}$ is continuous on
$\mathbb{Z}_{p}$ for $|s|_{p}<R_{p}=p^{(p-2)/(p-1)}$, following
the same reasoning as in Proposition~\ref{delta1}, we have the
following result.

\begin{proposition}~\label{delta2} For $(s,\chi,x)\in
B_{1}\times B_{2}\times B_{3}$,
we have $$a_{n}(s,\chi,x)= \sum_{\substack{d|n\\
(p,d)=1}} d^{-1}\chi(x+d)\langle x+d\rangle^{s}.$$
\end{proposition}

\begin{proposition}~\label{s2} Let $A_{1}=\mathbb{Z}_{p}$ and $A_{2}=\{\omega^{u}~|~
u~\textrm{is even}\}$. For $(s,u)\in X=A_{1}\times A_{2}$, we have
$a_{n}(s,\omega^{u},0)=a_{n}(G_{s,u}^{*}).$
\end{proposition}
\begin{proof} By Proposition~\ref{delta2}, Eq.~(\ref{eq1}), we have
\begin{equation}\begin{aligned}a_{n}(s,\omega^{u},0)&=\sum_{\substack{d|n\\
(p,d)=1}} d^{-1}\omega(d)^{u}\langle d\rangle^{s}
\\&=a_{n}(G_{s,u}^{*}).\end{aligned}\end{equation}
\end{proof}

In conclusion, we have extended the parameter space of $a_{n}(s,u)$
by the following graph (compare with Graph~(\ref{graph2}) above):

\begin{equation}\label{graph3}\begin{aligned}a_{n}(s,u),&~~ (s,u)\in
X\longrightarrow a_{n}(s,x),~~ (s,x)\in
C_{1}\times C_{2}\\&\downarrow\\  a_{n}(s,\omega^{u}),&~~ (s,\omega^{u})\in A_{1}\times A_{2}\\&\downarrow\\ a_{n}(s,\chi),&~~ (s,\chi)\in B_{1}\times B_{2}\\&\downarrow\\
a_{n}(s,\chi,x),&~~ (s,\chi,x)\in B_{1}\times B_{2}\times B_{3}
\end{aligned}\end{equation}

Notice that for $(s,u)\in X=\mathbb{Z}_{p}\times\mathbb{Z}/(p-1)\mathbb{Z}$ with $u$ even,  $G^{*}_{s,u}$ is Serre's $p$-adic family of
Eisenstein series
$$G^{*}_{s,u}=a_{0}(s,u)+\sum_{n=1}^{\infty}a_{n}(s,u)q^{n}.$$
Thus we can extend the parameter space of $G^{*}_{s,u}$ as
Graph~(\ref{graph2}) by extending the Fourier coefficients $a_{n}(s,u),n\geq 0$.

\begin{definition}~\label{df1}  For any area of $(s,x)\in C_{1}\times C_{2}$ indicated in Theorem~\ref{an1} above, we define
$$G^{*}(s,x)=a_{0}(s,x)+\sum_{n=1}^{\infty}a_{n}(s,x)q^{n},$$ where $a_{0}(s,x)$ and $a_{n}(s,x),n\geq 1,$ are defined in Eq. (\ref{a0sx}) and Definition \ref{ansx}, respectively.
\end{definition}

\begin{definition}~\label{df2}  For  $(s,\chi,x)\in B_{1}\times B_{2}\times B_{3}$, we define
$$G^{*}(s,\chi,x)=a_{0}(s,\chi, x)+\sum_{n=1}^{\infty}a_{n}(s,\chi, x)q^{n},$$ where $a_{0}(s,\chi, x)$ and $a_{n}(s,\chi,x),n\geq 1,$ are defined in Eq. (\ref{de1}) and Definition \ref{de2}, respectively.
\end{definition}

From Propositions \ref{s1} and \ref{s2} above, we have the following result.

\begin{proposition}\label{sg} Let $A_{1}=\mathbb{Z}_{p}$ and $A_{2}=\{\omega^u~|~
u~\textrm{is even}\}$. For $(s,u)\in X=A_{1}\times A_{2}$, we have
$G^{*}(s,\omega^{u},0)=G^{*}_{s,u},$ where $G^{*}_{s,u}$ is the $p$-adic family of Eisenstein
series \`{a} la Serre.
\end{proposition}

\begin{remark}\label{remark2} Our methods for extending $a_{n}(s,u),n\geq 0$, are based on Cohen's~\cite[Chapter 11]{Co2} and  Tangedal-Young's~\cite{TP} constructions of $p$-adic Hurwitz zeta functions. From Remark \ref{remark}, we may call above extensions of $p$-adic family of Eisenstein series $G^{*}_{s,u}$ (see Definitions~\ref{df1} and \ref{df2}) the $p$-adic analogue of Weil's elliptic functions (\ref{weil}).\end{remark}

\section{Power series expansions}

In this section, we show that the power series expansion  of Weil's elliptic functions (\ref{weil2}) also exists in the $p$-adic case.

 First we recall a result on the expansion of $p$-adic Hurwitz zeta functions by Tangedal and Young~\cite{TP} which leads to an expansion of the constant term
 $a_{0}(s,x)$ for $x\in\mathbb{C}_{p}^{*}$ and $|x|_{p}>1$.

 \begin{theorem}\label{TP1}Suppose $x\in\mathbb{C}_{p}^{*}$ and $|x|_{p}>1$. Then there is  an identity of analytic functions:
 \begin{equation}\label{G1} a_{0}(s,x)=\frac{1}{2}\zeta_{p}(1-s,x)=x\langle x \rangle^{s-1}\sum_{j=0}^{\infty}\binom{s}{j}\left(-\frac{B_{j}}{2s}\right)x^{-j}
 \end{equation}
 on the disc $|s-1|_{p} < p^{-1/(p-1)}|\rm{log}$$_{p}\langle x\rangle|_{p}^{-1}$. If in addition $x$ is chosen to lie in a finite extension $K$ of $\mathbb{Q}_{p}$ whose ramification index
 over $\mathbb{Q}_{p}$ is less than $p-1$, then this formula is valid for $s\in\mathbb{C}_{p}\backslash \{0\}$ such that $|s-1|_{p} < |\pi|_{p}^{-1}p^{-1/(p-1)}$,
 where $(\pi)$ is the maximal ideal of the ring of integers $O_{K}$ of $K$.\end{theorem}
  \begin{proof} By \cite[Theorem 4.1]{TP}, we have
 \begin{equation}~\label{TPe} \zeta_{p}(s,x)=\frac{x\langle x \rangle^{-s}}{s-1}\sum_{j=0}^{\infty}\binom{1-s}{j}B_{j}x^{-j},\end{equation}
combining with Eq.~(\ref{a0sx}), we have
 \begin{equation}
 \begin{aligned}
 a_{0}(s,x)&=\frac{1}{2}\zeta_{p}(1-s,x)\\&=x\langle x \rangle^{s-1}\sum_{j=0}^{\infty}\binom{s}{j}\left(-\frac{B_{j}}{2s}\right)x^{-j}.
 \end{aligned}
 \end{equation}
 \end{proof}
 \begin{remark} The above expansion~(\ref{TPe}) is in fact the $p$-adic analogue of the corresponding expansion for the classical Hurwitz zeta functions in \cite[p.~7]{Weil}.
 \end{remark}

 Now we show that $a_{n}(s,x), n\geq 1$, also have the similar expansions.

 \begin{theorem}\label{TP2}Suppose $x\in\mathbb{C}_{p}^{*}$ and $|x|_{p}>1$. Let $\sigma_{j}^{*}(n)=\sum_{\substack{d|n\\
(p,d)=1}}d^{j}$ and $n\geq 1$. Then there is an identity of analytic functions:
 \begin{equation}\label{G2} a_{n}(s,x)=x\langle x\rangle^{s-1}
\sum_{j=0}^{\infty}\binom{s}{j}\sigma_{j-1}^{*}(n)x^{-j},
 \end{equation} on the disc $|s-1|_{p} < p^{-1/(p-1)}|\rm{log}$$_{p}\langle x\rangle|_{p}^{-1}$. If in addition $x$ is chosen to lie in a finite extension $K$ of $\mathbb{Q}_{p}$ whose ramification index
 over $\mathbb{Q}_{p}$ is less than $p-1$, then this formula is valid for $s\in\mathbb{C}_{p}$ such that $|s-1|_{p} < |\pi|_{p}^{-1}p^{-1/(p-1)}$,
 where $(\pi)$ is the maximal ideal of the ring of integers $O_{K}$ of $K$. \end{theorem}
\begin{proof}  Since for $x\in\mathbb{C}_{p}^{*}$ and $|x|_{p}>1$, we have
\begin{equation}~\label{exp} \begin{aligned}\langle x+a \rangle^{s}&=\langle x\rangle^{s}
\left(1+\frac{a}{x}\right)^{s}\\&=\langle x\rangle^{s}\sum_{j=0}^{\infty}\binom{s}{j}a^{j}x^{-j}.\end{aligned}\end{equation}
From Proposition \ref{delta1} and (\ref{exp}), we have
\begin{equation}
 \begin{aligned}
a_{n}(s,x)&=\sum_{\substack{d|n\\
(p,d)=1}} d^{-1}\omega_{v}(x+d)\langle x+d\rangle^{s}\\&=\sum_{\substack{d|n\\
(p,d)=1}} d^{-1}\omega_{v}(x)\langle x\rangle^{s}
\sum_{j=0}^{\infty}\binom{s}{j}d^{j}x^{-j}\\&=\sum_{\substack{d|n\\
(p,d)=1}} x\langle x\rangle^{s-1}
\sum_{j=0}^{\infty}\binom{s}{j}d^{j-1}x^{-j}
\\&=x\langle x\rangle^{s-1}
\sum_{j=0}^{\infty}\binom{s}{j}\sigma_{j-1}^{*}(n)x^{-j}.
\end{aligned}
\end{equation}
\end{proof}

Combining Theorems ~\ref{TP1} and ~\ref{TP2}, we obtain the following $p$-adic analogue of
Weil's expansion (\ref{weil2}).

\begin{theorem}\label{power} Suppose $x\in\mathbb{C}_{p}^{*}$ and $|x|_{p}>1$. Let $\sigma_{j}^{*}(n)=\sum_{\substack{d|n\\
(p,d)=1}}d^{j}$ and $n\geq 1$. Let $$G^{*}(s,x)=a_{0}(s,x)+\sum_{n=1}^{\infty}a_{n}(s,x)q^{n}$$ be the $p$-adic Eisenstein family defined in Definition~\ref{df1} and $$e_{p,j}(s)=-\frac{B_{j}}{2s}+ \sum_{n=1}^{\infty}\sigma_{j-1}^{*}(n)q^{n}.$$ Then there is an identity of analytic functions:
\begin{equation} G^{*}(s,x)=x\langle x\rangle^{s-1}
\sum_{j=0}^{\infty}\binom{s}{j}e_{p,j}(s)x^{-j},
 \end{equation}on the disc $|s-1|_{p} < p^{-1/(p-1)}|\rm{log}$$_{p}\langle x\rangle|_{p}^{-1}$. If in addition $x$ is chosen to lie in a finite extension $K$ of $\mathbb{Q}_{p}$ whose ramification index
 over $\mathbb{Q}_{p}$ is less than $p-1$, then this formula is valid for $s\in\mathbb{C}_{p}\setminus\{0\}$ such that $|s-1|_{p} < |\pi|_{p}^{-1}p^{-1/(p-1)}$,
 where $(\pi)$ is the maximal ideal of the ring of integers $O_{K}$ of $K$. \end{theorem}
 \begin{proof} Since \begin{equation}~\label{G} G^{*}(s,x)=a_{0}(s,x)+\sum_{n=1}^{\infty}a_{n}(s,x)q^{n},\end{equation} substituting the equations (\ref{G1}) and (\ref{G2}) to (\ref{G}), we have
 \begin{equation}
 \begin{aligned}
 G^{*}(s,x)&=a_{0}(s,x)+\sum_{n=1}^{\infty}a_{n}(s,x)q^{n}\\&=x\langle x \rangle^{s-1}\sum_{j=0}^{\infty}\binom{s}{j}\left(-\frac{B_{j}}{2s}\right)x^{-j}+\sum_{n=1}^{\infty}\left(x\langle x\rangle^{s-1}\sum_{j=0}^{\infty}\binom{s}{j}\sigma_{j-1}^{*}(n)x^{-j}\right)q^{n}\\&=x\langle x \rangle^{s-1}\sum_{j=0}^{\infty}\binom{s}{j}\left(-\frac{B_{j}}{2s}\right)x^{-j}+ x\langle x\rangle^{s-1}\sum_{j=0}^{\infty}\binom{s}{j}\left(\sum_{n=1}^{\infty}\sigma_{j-1}^{*}(n)q^{n}\right)x^{-j} \\&=x\langle x \rangle^{s-1}\sum_{j=0}^{\infty}\binom{s}{j}\left(-\frac{B_{j}}{2s}+\sum_{n=1}^{\infty}\sigma_{j-1}^{*}(n)q^{n}\right)x^{-j}
 \\&=x\langle x\rangle^{s-1}
\sum_{j=0}^{\infty}\binom{s}{j}e_{p,j}(s)x^{-j},
\end{aligned}
 \end{equation}
 where $e_{p,j}(s)=-\frac{B_{j}}{2s}+ \sum_{n=1}^{\infty}\sigma_{j-1}^{*}(n)q^{n}.$
 \end{proof}

\section*{Acknowledgment}
The authors are enormously grateful to the anonymous referee for his/her very careful reading of this paper, and for his/her many valuable and detailed
suggestions.

This work was supported by the Kyungnam University Foundation Grant, 2014.

\bibliography{central}

\end{document}